\documentclass[a4paper,12pt]{article}

\usepackage{setspace}
\usepackage[utf8]{inputenc}
\usepackage{bm}
\usepackage{amssymb}
\usepackage{amsmath}
\usepackage{amsfonts}
\usepackage{amstext}
\usepackage{multicol}
\usepackage{mathdots}
\usepackage{graphicx}
\usepackage{url}
\usepackage{hyperref}
\usepackage{mathrsfs}
\usepackage{lscape}
\usepackage{enumerate}
\usepackage{amsthm}
\usepackage{mathtools}
\usepackage{cite}
\usepackage{multirow}

\def\modd#1 #2{#1\ \mbox{\rm (mod}\ #2\mbox{\rm )}}
\DeclareMathOperator{\mex}{mex}

\newcommand{\Mod}[1]{\ (\mathrm{mod}\ #1)}
\newcommand{\mexover}[1]{ \overline{\mex}({#1})}
\newcommand{\smexover}[1]{\sigma \overline{\mex}({#1})}

\newcommand{\mexoverover}[1]{ \widehat{\mex}({#1})}
\newcommand{\smexoverover}[1]{\sigma \widehat{\mex}({#1})}

\newtheorem{theorem}{Theorem}[section]

\newtheorem{definition}{Definition}
\newtheorem{example}{Example}

\newtheorem{lemma}{Lemma}

\newtheorem{proposition}{Proposition}

\title{On Minimal Excludant over Overpartitions}
\author{Judy Ann L. Donato\thanks{Institute of Mathematics, University of the Philippines, Diliman, Quezon
City 1101, Philippines; \texttt{jadonato@math.upd.edu.ph}}}
\date{}
\begin{document}
\onehalfspacing
\maketitle

\begin{abstract}
A partition of a positive integer $n$ is a non-increasing sequence of positive integers which sum to $n$. A recently studied aspect of
partitions is the minimal excludant of a partition, 
which is defined to be the smallest positive integer that is not a part of the partition. In 2024, Aricheta and Donato studied the minimal excludant of the non-overlined parts of an overpartition, where
an overpartition of $n$ is a partition of $n$ in which the first occurrence of a number may be overlined. In this research, we explore two other definitions of the minimal excludant of an overpartition: (i) considering only the overlined parts, and (ii) considering both the overlined and non-overlined parts. We discuss the  combinatorial, asymptotic, and arithmetic
properties of the corresponding $\sigma$-function, which gives the sum of the minimal excludants over all overpartitions.  \\
\ \ \ \ \


\noindent \textbf{AMS\ Classification: }05A17, 11F11, 11F20, 11P83
\end{abstract}


\newpage
\section{Introduction}
The minimal excludant (mex) of a subset $S$ of a well-ordered set $U$ is the smallest value in $U$ that is not in $S$. The history of the minimal excludant goes way back in the 1930's when it was first used in combinatorial game theory by Sprague and Grundy \cite{sprague}, \cite{grundy}.

In 2006, Grabner and Knopfmacher \cite{Grabner} first applied the notion of  minimal excludant to integer partitions. They studied a new statistic for integer partitions called the smallest gap.  This was considered again in 2019 (using a different terminology), when Andrews and Newman \cite{andrewsnewman} defined the minimal excludant of an integer partition $\pi$, denoted $\mex(\pi)$, as the smallest positive integer that is not a part of $\pi$. Moreover, they also introduced the arithmetic function
\[\sigma \text{mex}(n):=\displaystyle \sum_{\pi \in \mathcal{P}(n)} \text{mex}(\pi),\]
where $\mathcal{P}(n)$ is the set of all partitions of $n$.

They proved that $\sigma \text{mex}(n)$ is equal to  $D_2(n)$, which is the number of partitions of $n$ into distinct parts using two colors. Moreover, they have shown that $\sigma \text{mex}(n)$ is almost always even and is odd exactly when $n=j(3j\pm 1)$ for some $j \in \mathbb{N}$.

The notion of minimal excludant can also be applied to other types of partitions. In particular, this study focuses on overpatitions. An overpartition of a positive integer $n$ is a non-increasing sequence of natural numbers whose sum is $n$ in which the first occurrence  of a number may be overlined. We denote by $\overline{p}(n)$ the number of overpartitions of $n$. For example, $\overline{p}(3)=8$ since there are 8 overpartitions of $3$ which are: 
\[3, \overline{3}, 2+1, \overline{2}+1, 2+\overline{1}, \overline{2}+\overline{1}, 1+1+1, \overline{1}+1+1.\]
Overpartitions are interesting to look at because both the properties of the classical partition (non-overlined parts) and the partitions into distinct parts (overlined parts) are present in this kind of partition.

In 2024, Aricheta and Donato \cite{ArichetaDonato}  applied the concept of minimal excludant of partitions to overpartitions and studied the minimal excludant of the non-overlined parts of an overpartition, denoted $\mexover{\pi}$, which is defined to be the smallest positive integer that is not a part of the \textit{non-overlined} parts of $\pi$.  For a positive integer $n$, denote the sum of $\mexover{\pi}$ over all overpartitions $\pi$ of $n$ as $\smexover{n}:$
\[\smexover{n}=\displaystyle \sum_{\pi \in \overline{\mathcal{P}}(n)} \mexover{\pi}, \]
where $\overline{\mathcal{P}}(n)$ is the set of all overpartitions of $n$. Moreover, we set $\smexover{0}=1$.

They proved that $\smexover{n}$ is equal to $D_3(n)$, the number of partitions of $n$ into distinct parts of three color which is analogous to the results of Andrews and Newman. They also proved that  $\smexover{n}$ is
almost always even and is odd exactly when $n=\frac{j(j+1)}{2}$ for some $j \in \mathbb{Z}$. They also derived an asymptotic formula for $\smexover{n}$ using Ingham's Tauberian Theorem on partitions.

Note that in the definition provided in \cite{ArichetaDonato}, the minimal excludant of an overpartition is based only on the non-overlined parts. This study extends the concept by introducing two alternative definitions of the minimal excludant for overpartitions: (i) considering only the overlined parts, and (ii) considering both the overlined and non-overlined parts.

First, we have the following definition for the minimal excludant of the overlined parts.

\begin{definition} \label{def1}
     The \textbf{minimal excludant of the overlined parts of an overpartition} $\pi$, denoted $\mexoverover{\pi}$, is the smallest positive integer that is not a part of the \textbf{overlined} parts of $\pi$.  For a positive integer $n$, denote the sum of $\mexoverover{\pi}$ over all overpartitions $\pi$ of $n$ as $\smexoverover{n}:$
\[\smexoverover{n}=\displaystyle \sum_{\pi \in \overline{\mathcal{P}}(n)} \mexoverover{\pi}. \]
Moreover, we set $\smexoverover{0}=1$..
\end{definition}

\newpage
As an example, the table below shows all overpartitions of $n=3$ and their corresponding minimal excludants of the overlined parts.

    \begin{center}
\begin{tabular}{ |c|c| } 
 \hline
 $\pi$ & $\mexoverover{\pi}$\\
 \hline
  $3$ &  $1$ \\ 
$\overline{3}$ &  $1$ \\ 
  $2 + 1$ &  $1$ \\
  $\overline{2} + 1$ &  $1$ \\
   $2 + \overline{1}$ &  $2$ \\
    $\overline{2} + \overline{1}$ &  $3$\\
    $1+1+1 $ &  $1$\\
   $\overline{1}+1+1$ &  $2$\\
 \hline 
\end{tabular}
\end{center}
Thus, $\smexoverover{3}= 12$.

We prove that the generating function for $\smexoverover{n}$ is related to the Ramanujan's $q$-series 
\[\sigma(q)=\displaystyle\sum_{m=0}^{\infty}\dfrac{ q^{{m+1} \choose 2}}{{(-q;q)_m}}.\] This series first appeared in Ramanujan's Lost Notebook: Part V \cite{Andrews6}. We have the following results.

\begin{theorem} \label{gen2}
For all positive integers $n$, we have 
\[\displaystyle \sum_{n=0}^\infty \smexoverover{n} q^n =\dfrac{(-q;q)_\infty}{(q;q)_\infty}\displaystyle \sum_{m=0}^\infty \dfrac{mq^{m \choose 2}}{(-q;q)_m}=\overline{P}(q)\sigma(q),\]
where $\overline{P}(q)$ is the generating function for  $\overline{p}(n)$ and $\sigma(q)$ is the Ramanujan's $q$-series.
\end{theorem}

We then prove that $\smexoverover{n}$ is almost always even and derive an asymptotic formular for $\smexoverover{n}$.

\begin{theorem} \label{parity2}
We have 
\[\displaystyle \lim_{X \to +\infty} \dfrac{\#\{n\leq X: \smexoverover{n} \equiv \modd{0} {2}\}}{X}=1.\]
\end{theorem}

\begin{theorem} \label{asym2}
We have
    \[\smexoverover{n} \sim \dfrac{e^{\pi{\sqrt{n}}}}{4n}\] \text{ as } $ n \rightarrow \infty.$ 
\end{theorem}

We now present the following definition for the minimal excludant of both the overlined and non-overlined parts, or we simply call as the minimal excludant of an overpartition.

\begin{definition}
 The \textbf{minimal excludant of an overpartition} $\pi$, denoted $\mexover{\pi}$, is the smallest positive integer that is not a part of \textbf{both the overlined and non-overlined parts} of $\pi$.  For a positive integer $n$, denote the sum of $\mexover{\pi}$ over all overpartitions $\pi$ of $n$ as $\smexover{n}:$
\[\smexover{n}=\displaystyle \sum_{\pi \in \overline{\mathcal{P}}(n)} \mexover{\pi}. \] Moreover, we set $\smexover{0}=1$.
\end{definition}

Note that we used the same notation used in \cite{ArichetaDonato} for the minimal excludant of the non-overlined parts of an overpartition for the definition above. This is because it is more natural to use the notations $\mexover{n}$ and $\smexover{n}$ for the definition of the minimal excludant of both parts.

The table below shows all overpartitions of $n=3$ and their corresponding minimal excludants.
\begin{center}
\begin{tabular}{ |c|c| } 
 \hline
 $\pi$ & $\mexover{\pi}$\\
 \hline
 $3$ & $1$ \\ 
$\overline{3}$ & $1$ \\ 
  $2 + 1$ & $3$\\
  $\overline{2} + 1$ & $3$ \\
   $2 + \overline{1}$ & $3$\\
     $\overline{2} + \overline{1}$ & $3$\\
   $1+1+1$ & $2$\\
   $\overline{1}+1+1$ & $2$\\
 \hline 
\end{tabular}
\end{center}
Thus, $\smexover{3}= 18$.

We initially observe the parity of $\smexover{n}$ and prove that $\smexover{n}$ is always even using its combinatorial properties.

\begin{theorem} \label{parity3} For a positive integer $n$, $\smexover{n}$ is always even.
\end{theorem}

Then, we show that the generating function of $\smexover{n}$ can be related to the unilateral basic hypergeometric series ${}_1\phi_1(q;-q;q,-2q)$ as follows.

\begin{theorem}  \label{gen3}
For all positive integers $n$, we have 
\[\displaystyle \sum_{n=0}^\infty \smexover{n}q^n=\dfrac{(-q;q)_\infty}{(q;q)_\infty}\displaystyle \sum_{m=0}^\infty \dfrac{2^mq^{m+1 \choose 2}}{(-q;q)_m} = \overline{P}(q){}_1\phi_1(q;-q;q,-2q),\]
where $\overline{P}(q)$ is the generating function for $\overline{p}(n)$ and ${}_1\phi_1(q;-q;q,-2q)$ is the unilateral basic hypergeometric series
\[\displaystyle \sum_{n=0}^\infty \dfrac{2^n\cdot q^{{n+1} \choose 2}}{(-q;q)_n}.\]
\end{theorem}


\section{Minimal Excludant of the Overlined Parts of an Overpartition}
\subsection{Generating Function of $\smexoverover{n}$}

\begin{proof}[Proof of Theorem \ref{gen2}]
Let $p^{\widehat{\mex}}(m,n)$ be the number of overpartitions $\pi$ of $n$ with $\mexoverover{\pi}=m.$  
Then we have the following double series $M(z,q)$ in which the coefficient of $z^mq^n$ is $p^{\widehat{\mex}}(m,n)$:

\begin{align*}
M(z,q)&:=\displaystyle\sum_{n=0}^\infty \displaystyle\sum_{m=1}^\infty p^{\widehat{\mex}}(m,n)z^mq^n\\
&=\displaystyle \sum_{m=1}^{\infty} z^{m}\cdot q^1\cdot q^2 \cdots \cdot q^{m} \cdot \frac{\displaystyle\prod_{n=m+1}^\infty (1+q^n)}  {\displaystyle\prod_{n=1}^\infty (1-q^n)}\\
&=\dfrac{(-q;q)_\infty}{(q;q)_\infty} \displaystyle \sum_{m=1}^{\infty}\dfrac{z^{m} \cdot q^{m \choose 2}}{{(-q;q)_m}}.
\end{align*}


\noindent Thus,
\begingroup
\allowdisplaybreaks
\begin{align*}
    \displaystyle \sum_{n \geq 0} \smexoverover{n} q^n &= \dfrac{\partial}{\partial z}\Big|_{z=1} M(z,q)\\
    &=\dfrac{(-q;q)_\infty}{(q;q)_\infty} \displaystyle\sum_{m=1}^{\infty} \dfrac{m \cdot q^{m \choose 2}}{{(-q;q)_m}}\\
     &=\overline{P}(q)\displaystyle\sum_{m=1}^{\infty} \dfrac{m \cdot q^{m \choose 2}}{{(-q;q)_m}}\\
      &=\overline{P}(q)\displaystyle\sum_{m=1}^{\infty} \dfrac{m \cdot q^{m \choose 2}}{{(-q;q)_{m-1}}}\cdot\dfrac{1}{1+q^m}\\
       &=\overline{P}(q)\displaystyle\sum_{m=1}^{\infty} \dfrac{m \cdot q^{m \choose 2}}{{(-q;q)_{m-1}}} \left(1-\dfrac{q^m}{1+q^m} \right)\\
        &=\overline{P}(q) \left[\displaystyle\sum_{m=1}^{\infty} \ \dfrac{m \cdot q^{m \choose 2}}{{(-q;q)_{m-1}}} - \dfrac{m \cdot q^{{m+1} \choose 2}}{{(-q;q)_m}}\right]\\
         &=\overline{P}(q)\left[\displaystyle\sum_{m=0}^{\infty} \ \dfrac{(m+1) \cdot q^{{m+1} \choose 2}}{{(-q;q)_{m}}} - \displaystyle\sum_{m=0}^{\infty}\dfrac{m \cdot q^{{m+1} \choose 2}}{{(-q;q)_m}}\right]\\
         &=\overline{P}(q)\displaystyle\sum_{m=0}^{\infty}\dfrac{ q^{{m+1} \choose 2}}{{(-q;q)_m}}\\
    &=\overline{P}(q)\sigma(q).
 \end{align*}
\endgroup
\end{proof}


\subsection{Parity of $\smexoverover{n}$}
\begin{lemma}\label{overeven}
    The overpartition number $\overline{p}(n)$ is even for any positive integer n.
\end{lemma}
\begin{proof}
    Note that 
\[\overline{P}(q)=\displaystyle \sum_{n=1}^\infty \overline{p}(n)q^n =\dfrac{(-q;q)_\infty}{(q;q)_\infty} = \dfrac{(q^2;q^2)_\infty}{(q;q)^2_\infty} .\]
By the binomial theorem, $(q^2;q^2)_\infty \equiv (q;q)^2_\infty \Mod 2.$ This implies that $\overline{p}(n)$ is always even for all positive integers $n$.
\end{proof}

\begin{proof}[Proof of Theorem \ref{parity2}]
It was shown by Andrews et al. in \cite{Andrews2} that if $\sigma(q)=\sum_{n=0}^\infty S(n)q^n,$ then $S(n)$ is almost always zero. Hence, combining this with Lemma \ref{overeven} implies that $\smexoverover{n}$ is almost always even.
\end{proof}


\subsection{Asymptotic Formula for $\smexoverover{n}$}

To derive an asymptotic formula for $\smexoverover{n}$, we use the following asymptotic result by Ingham \cite{ingham} about the coefficients of a power series.

\begin{proposition} \label{asymingham}
  Let $A(q)= \sum_{n=0}^{\infty} a(n)q^n$ be a power series with radius of convergence equal to 1. Assume that $\{a(n)\}$ is a weakly increasing sequence of nonnegative real numbers. If there are constants $\alpha, \beta \in \mathbb{R}$, and $C>0$ such that
\[A(e^{-t})\sim \alpha t^{\beta}e^{\frac{C}{t}}, \text{ as } t \rightarrow 0^{+} \] Then 

\[a(n) \sim \dfrac{\alpha}{2\sqrt{\pi}}\dfrac{C^{\frac{2\beta+1}{4}}}{n^{\frac{2\beta+3}{4}}}e^{2\sqrt{Cn}}, \text{ as } n \rightarrow \infty.\]
\end{proposition}

\begin{proof}[Proof of Theorem \ref{asym2}]
Note that $\{\smexoverover{n}\}$ is an increasing sequence of nonnegative real numbers, since $\{\overline{p}(n)\}$ is also an increasing sequence of nonnegative real numbers.\\

\noindent Let $A(q)=\dfrac{(-q;q)_\infty}{(q;q)_\infty}\sigma(q)$, where $a(n)=\smexoverover{n}$ as in Proposition \ref{asymingham}. From \cite{bhoria},
\begin{align}\dfrac{1}{(e^{-t}; e^{-t})_\infty} \sim \sqrt{\dfrac{t}{2\pi}} e^{\frac{\pi^2}{6t}} \text{ as } t \rightarrow 0^+. \label{1}
\end{align}
Moreover, we use the following identity of Euler
\begin{align}
(-q;q)_{\infty}=\dfrac{1}{(q;q^2)_\infty}=\dfrac{(q^2;q^2)_\infty}{(q;q)_\infty}.\label{2}
\end{align}

\noindent By (\ref{1}) and (\ref{2}), as $t \rightarrow 0^{+}$,
\[\dfrac{(-e^{-t};e^{-t})_\infty}{(e^{-t};e^{-t})_\infty} = \dfrac{(e^{-2t};e^{-2t})_\infty}{(e^{-t};e^{-t})^2_\infty}\sim \frac{\frac{t}{2\pi}  e^{\frac{2\pi^2}{6t}}}{\sqrt{\frac{2t}{2\pi}} e^{\frac{\pi^2}{12t}}}=\dfrac{\sqrt{t}}{2\sqrt{\pi}}e^{\frac{\pi^2}{4t}}.\] 
Moreover, from \cite{Zagier} (p.6), 
\[\sigma(e^{-t})=2-2t+5t^2-\dfrac{55}{3}t^3+\dfrac{1073}{12}t^4-\dfrac{32671}{60}t^5+\dfrac{286333}{72}t^6-\cdots.\]
Hence, as $t \rightarrow 0^{+}$,
\begin{align}
A(e^{-t}) \sim \dfrac{1}{\sqrt{\pi}}t^{1/2}e^{\frac{\pi^2}{4t}}. \label{3}
\end{align}

\noindent Take $\alpha=\frac{1}{\sqrt{\pi}}, \beta=\frac{1}{2}$ and $C=\frac{\pi^2}{4}$, by Proposition \ref{asymingham},
\[\smexoverover{n} \sim \dfrac{\frac{1}{\sqrt{\pi}}}{2\sqrt{\pi}} \dfrac{\left(\frac{\pi^2}{4}\right)^{1/2}}{n}e^{2\sqrt{\frac{\pi^2}{4}n}} = \dfrac{e^{\pi{\sqrt{n}}}}{4n}\]
as $n \rightarrow \infty$.
\end{proof}

\section{Minimal Excludant of an Overpartition}
\subsection{Parity of $\smexover{n}$}

 Before we derive the generating function for $\smexover{n}$, we first discuss the parity of $\smexover{n}$ and observe some of its combinatorial properties.
\begin{theorem}\label{2m}
    Let $\mathcal{S}$ be a multiset of positive integers. Let $m$ be the number of distinct integers in $\mathcal{S}$ and let n be the sum of elements of the $\mathcal{S}$. There are $2^m$ overpartitions of $n$ that can be formed using the integers in $S$.
\end{theorem}

\begin{proof}
Partition $\mathcal{S}$ into $m$ multisubsets, $S_1,...,S_m$ where we group the same integers. Note that in an overpartition, only the first occurrence of a number may be overlined, hence for each multiset $S_i$, there are only two ways to write the  equal integers in the multiset, one with all parts not overlined and the other with one part overlined. Thus, there are $2^m$ ways to form an overpartition using the integers in $\mathcal{S}$ as parts.
\end{proof}

\begin{example}
    Suppose we have $\mathcal{S}=\{5,3,3,3,2,2\}$, then $m=3$ and $n=18$. Moreover, there are exactly 2 overpartitions of 5 whose parts are elements of  $S_1=\{5\}$, namely $5$ and $\overline{5}$, 2 overpartitions of 9 whose parts are elements of   $S_2=\{3,3,3\}$, namely $3+3+3$ and $\overline{3}+3+3$, and 2 overpartitions of 4 whose parts are elements of  $S_3=\{2,2\}$, namely $2+2$ and $\overline{2}+2$. Hence, there are $2\cdot2\cdot2=2^3=8$ overpartitions of $18$ whose parts are elements of $\mathcal{S}$. 
\end{example}

\begin{proof}[Proof of Theorem \ref{parity3}]
Let $\pi$ be an overpartition. Let $\pi'$ be the ordinary partition obtained by removing the overlines in the overlined parts of $\pi$. For example, if $\pi=5+\overline{3}+3+\overline{2}+1$, then $\pi'=5+3+3+2+1$.
 
We  divide the set  of overpartitions of $n$ into subsets $\overline{\mathcal{P}}_1,...,\overline{\mathcal{P}}_k$  such that the overpartitions with the same $\pi'$  are grouped together. Let $m_1,...,m_k$ be the number of distinct parts of the corresponding $\pi'$ for each subset $\overline{\mathcal{P}}_1,...,\overline{\mathcal{P}}_k$, respectively. By Theorem \ref{2m}, the number of elements in each subset is $2^{m_1},...,2^{m_k} $, where $m_1,...,m_k \geq 1$. Moreover, the overpartitions in the same subset have equal minimal excludants. Thus, if the minimal excludant of any overpartition in the subset $\overline{\mathcal{P}}_1,...,\overline{\mathcal{P}}_k$ are $a_1,...,a_k$, respectively, then
\[\smexover{n}=2^{m_1}\cdot a_1+\cdots+2^{m_k}\cdot a_k,\]
proving that $\smexover{n}$ is always even.
\end{proof}
\noindent Here, we provide an illustration of Theorem \ref{parity3}. We divide the overpartitions of $n=4$ into the following subsets: 
\begin{align*}
    \overline{\mathcal{P}}_1&=\{4, \overline{4}\}\\ \overline{\mathcal{P}}_2&=\{3+1, \overline{3}+1,3+\overline{1},\overline{3}+\overline{1}\}\\ \overline{\mathcal{P}}_3&=\{2+2, \overline{2}+2\}\\ \overline{\mathcal{P}}_4&=\{2+1+1, \overline{2}+1+1,2+\overline{1}+1,\overline{2}+\overline{1}+1\}\\
    \overline{\mathcal{P}}_5&=\{1+1+1+1, \overline{1}+1+1+1\}.
\end{align*}

Notice that the overpartitions in each subset have the same $\pi'$ and minimal excludant as shown in the table below.\\
\begin{center}
\begin{tabular}{ |c|c|c|} 
 \hline
 $\pi$ & $\pi'$ & $\mexover{\pi}$\\
 \hline
 $4$ & \multirow{2}{*}{$4$} & \multirow{2}{*}{$1$} \\ 
$\overline{4}$ &  &  \\ 
\hline
$3 + 1$ &  \multirow{4}{*}{$3 + 1$} & \multirow{4}{*}{$2$} \\
$\overline{3} + 1$ & &\\
$3 + \overline{1}$ & &\\
$\overline{3} + \overline{1}$ & &\\
\hline
$2 + 2$ & \multirow{2}{*}{$2+2$}  & \multirow{2}{*}{$1$} \\
$\overline{2} + 2$ && \\
\hline
$2+1+1$ & \multirow{4}{*}{$2+1+1$} & \multirow{4}{*}{$3$}\\
$\overline{2} + 1+1$&&\\
$2+\overline{1}+1$&& \\
$\overline{2} + \overline{1}+1$ & &\\
\hline  
$1+1+1+1$ &  \multirow{2}{*}{$1+1+1+1$} &  \multirow{2}{*}{$2$} \\
$\overline{1}+1+1+1$ & &\\
 \hline 
\end{tabular}
\end{center}
Thus, $\smexover{4}=2\cdot 1+2^2\cdot 2+ 2\cdot1+2^2\cdot3+2\cdot 2=28$.

\subsection{Generating Function of $\smexover{n}$}

\begin{proof}[Proof of Theorem \ref{gen3}]
Let $p^{\overline{\text{mex}}}(m,n)$ be the number of overpartitions $\pi$ of $n$ with $\mexover{\pi}=m.$
Then we have the following double series $M(z,q)$ in which the coefficient of $z^mq^n$ is $p^{\overline{\text{mex}}}(m,n)$:
\begin{align*}
M(z,q)&:=\displaystyle \sum_{n=0}^\infty \displaystyle\sum_{m=1}^\infty p^{\overline{\text{mex}}}(m,n)z^mq^n\\
&=\displaystyle \sum_{m=1}^{\infty} z^{m}\cdot 2^{m-1} \cdot q^1\cdot q^2 \cdots \cdot q^{m-1} \cdot \frac{\displaystyle\prod_{n=m+1}^\infty (1+q^n)}  {\displaystyle\prod_{\substack{n=1 \\ n \neq m}}^\infty (1-q^n)}\\
&=\dfrac{(-q;q)_\infty}{(q;q)_\infty} \displaystyle \sum_{m=1}^{\infty}z^{m} \cdot 2^{m-1} \cdot q^{m \choose 2}\cdot \dfrac{(1-q^m)}{(-q;q)_m}.
\end{align*}


\noindent Thus,
\begingroup
\allowdisplaybreaks
\begin{align*}
    \displaystyle \sum_{n \geq 0} \smexover{n} q^n&= \dfrac{\partial}{\partial z}\Big|_{z=1} M(z,q)\\
    &=\dfrac{(-q;q)_\infty}{(q;q)_\infty} \displaystyle\sum_{m=1}^{\infty}m\cdot 2^{m-1} \cdot q^{m \choose 2}\cdot \dfrac{(1-q^m)}{(-q;q)_m} \stepcounter{equation}\tag{\theequation}\label{orig} \\
    &=\dfrac{(-q;q)_\infty}{(q;q)_\infty} \left[\displaystyle\sum_{m=1}^{\infty}\dfrac{m\cdot 2^{m-1} \cdot q^{m \choose 2}}{{(-q;q)_m}}-\displaystyle\sum_{m=1}^{\infty}\dfrac{m \cdot 2^{m-1} \cdot q^{m \choose 2}q^m}{{(-q;q)_m}}\right]\\
    &=\dfrac{(-q;q)_\infty}{(q;q)_\infty} \left[\displaystyle\sum_{m=1}^{\infty}\dfrac{m \cdot 2^{m-1} \cdot  q^{m \choose 2}}{{(-q;q)_m}}-\displaystyle\sum_{m=1}^{\infty}\dfrac{m\cdot 2^{m-1} \cdot q^{m+1 \choose 2}}{{(-q;q)_m}}\right]\\
    &=\dfrac{(-q;q)_\infty}{(q;q)_\infty} \left[\displaystyle\sum_{m=1}^{\infty}\dfrac{m\cdot 2^{m-1} \cdot q^{m \choose 2}}{{(-q;q)_m}}-\displaystyle\sum_{m=1}^{\infty}\dfrac{(m-1)\cdot 2^{m-2} \cdot q^{m \choose 2}}{{(-q;q)_{m-1}}}\right]\\
     &=\dfrac{(-q;q)_\infty}{(q;q)_\infty} \left[\displaystyle\sum_{m=1}^{\infty}\dfrac{m\cdot 2^{m-1} \cdot q^{m \choose 2}}{{(-q;q)_m}}-\dfrac{(m-1)\cdot 2^{m-2} \cdot 
 q^{m \choose 2}}{{(-q;q)_{m-1}}}\right]\\
      &=\dfrac{(-q;q)_\infty}{(q;q)_\infty} \left[\displaystyle\sum_{m=1}^{\infty} \dfrac{2^{m-1} \cdot q^{m \choose 2}}{(-q;q)_{m}} \left(m-\dfrac{(m-1)(1+q^m)}{2}\right)\right]\\
    &=\dfrac{(-q;q)_\infty}{(q;q)_\infty} \left[\displaystyle\sum_{m=1}^{\infty} \dfrac{2^{m-1}\cdot q^{m \choose 2}}{(-q;q)_{m}} \left(\dfrac{m-mq^m+1+q^m}{2}\right)\right]\\
    &=\dfrac{1}{2}\dfrac{(-q;q)_\infty}{(q;q)_\infty} \displaystyle\sum_{m=1}^{\infty} \dfrac{m\cdot 2^{m-1} \cdot q^{m \choose 2} \cdot (1-q^m)}{(-q;q)_{m}}+\dfrac{(-q;q)_\infty}{(q;q)_\infty} \displaystyle\sum_{m=1}^{\infty} \dfrac{ 2^{m-2} \cdot q^{m \choose 2}}{(-q;q)_{m-1}}.
\end{align*}
\endgroup

\noindent Hence, from (\ref{orig}), we have
\[
\displaystyle \sum_{n \geq 0} \smexover{n} q^n = \dfrac{1}{2} \displaystyle \sum_{n \geq 0} \smexover{n} q^n + \dfrac{(-q;q)_\infty}{(q;q)_\infty} \displaystyle\sum_{m=1}^{\infty} \dfrac{ 2^{m-2} \cdot q^{m \choose 2}}{(-q;q)_{m-1}}.\]

\noindent Thus,
\begin{align*}
\displaystyle \sum_{n \geq 0} \smexover{n} q^n &=\dfrac{(-q;q)_\infty}{(q;q)_\infty} \displaystyle\sum_{m=1}^{\infty} \dfrac{2\cdot 2^{m-2} \cdot q^{m \choose 2}}{(-q;q)_{m-1}}\\
&= \dfrac{(-q;q)_\infty}{(q;q)_\infty} \displaystyle\sum_{m=1}^{\infty} \dfrac{2^{m-1} \cdot q^{m \choose 2}}{(-q;q)_{m-1}}\\
&= \dfrac{(-q;q)_\infty}{(q;q)_\infty} \displaystyle\sum_{m=0}^{\infty} \dfrac{2^{m} \cdot q^{m+1 \choose 2}}{(-q;q)_{m}}.
\end{align*}
Now,
\begin{align*}
{}_1\phi_1(q;-q;q,-2q)&= \displaystyle \sum_{n=0}^\infty \dfrac{(q;q)_n}{(-q,q;q)_n} \left((-1)^nq^{n \choose 2}\right)(-2q)^n\\
&= \displaystyle \sum_{n=0}^\infty \dfrac{(q;q)_n}{(-q;q)_n(q;q)_n} \left((-1)^nq^{n \choose 2}\right)(-2q)^n\\
&= \displaystyle \sum_{n=0}^\infty \dfrac{1}{(-q;q)_n} \cdot2^nq^{n \choose 2}\cdot q^n\\
&= \displaystyle \sum_{n=0}^\infty \dfrac{2^n\cdot q^{{n+1} \choose 2}}{(-q;q)_n}.
\end{align*}

\noindent Therefore,
\[\displaystyle \sum_{n \geq 0} \smexover{n} q^n=\dfrac{(-q;q)_\infty}{(q;q)_\infty} \displaystyle\sum_{m=0}^{\infty} \dfrac{2^{m} \cdot q^{m+1 \choose 2}}{(-q;q)_{m}}=\overline{P}(q){}_1\phi_1(q;-q;q,-2q).\]
\end{proof}

\noindent {\bf Declarations}

\noindent \textbf{Conflict of Interests.}  The author have no relevant financial or non-financial interests to disclose.

\end{document}